\documentclass[a4paper,oneside,reqno,12pt]{amsart}
\usepackage{amssymb,mathtools}
\usepackage{MnSymbol}
\usepackage{dsfont}

\usepackage[hidelinks]{hyperref}
\usepackage{nameref}
\hypersetup{colorlinks = true, urlcolor = blue, linkcolor = blue, citecolor = red}

\usepackage{enumitem}

\numberwithin{equation}{section}
\theoremstyle{plain}
\newtheorem{theorem}[equation]{Theorem}
\newtheorem{conjecture}[equation]{Conjecture}
\newtheorem{proposition}[equation]{Proposition}
\newtheorem{corollary}[equation]{Corollary}
\newtheorem{lemma}[equation]{Lemma}
\theoremstyle{definition}

\newtheorem{fact}[equation]{Fact}

\newcommand{\RR}{\mathbb{R}}
\newcommand{\ZZ}{\mathbb{Z}}

\newcommand{\NN}{\mathbb{N}}
\newcommand{\QQ}{\mathbb{Q}}
\newcommand{\XX}{\mathbb{X}}

\newcommand{\calB}{\mathcal{B}}
\newcommand{\calC}{\mathcal{C}}
\newcommand{\calM}{\mathcal{M}}

\newcommand{\calS}{\mathcal{S}}

\newcommand{\TOS}{T_*^{\rm os}}
\newcommand{\TC}{T_*^{\rm c}}
\newcommand{\TNC}{T_*^{\rm u}}

\newcommand{\COS}{\calC^{\rm os}}
\newcommand{\CCC}{\calC^{\rm c}}
\newcommand{\CNC}{\calC^{\rm u}}

\newcommand{\EOS}{E^{\rm os}_\lambda}
\newcommand{\EC}{E^{\rm c}_\lambda}
\newcommand{\ENC}{E^{\rm u}_\lambda}

\newcommand{\ind}[1]{{\mathds{1}_{{#1}}}}

\allowdisplaybreaks

\title[Sharp constants in Calder\'on transference]
{Sharp constants in inequalities admitting \\ the Calder\'on transference principle}

\author{Dariusz Kosz}
\address{\scriptsize Dariusz Kosz (\textnormal{dkosz@bcamath.org}) \newline
	Basque Center for Applied Mathematics, 48009 Bilbao, Spain 
\newline
	Wroc{\l}aw University of Science and Technology, 50-370 Wroc{\l}aw, Poland
}

\begin{document}

\begin{abstract}
The aim of this note is twofold. Firstly, we prove an abstract version of the Calder\'on transference principle for inequalities of admissible type in the general commutative multilinear and multiparameter setting. Such an operation does not increase the constants in the transferred inequalities. Secondly, we use the last information to study a certain dichotomy arising in problems of finding the best constants in the weak type $(1,1)$ and strong type $(p,p)$ inequalities for one-parameter ergodic maximal operators.


\end{abstract}

\maketitle

\section{Introduction}\label{S1}

\subsection{Historical background} The Calder\'on transference principle \cite{Ca} is a powerful tool in ergodic theory which allows one to transfer various quantitative results, such as inequalities of strong or weak type, from one specific dynamical system -- called \emph{canonical} -- to all systems of the same type. Somewhat paradoxically, its main application is related to a~purely qualitative property. Indeed, this kind of transference is widely used in verifying almost everywhere convergence of ergodic averages as the time parameter goes to infinity. 

The classical approach to the latter subject relies on the following two-step procedure:
\begin{itemize}
	\item finding a dense class of functions for which pointwise convergence holds,
	\item proving an appropriate maximal inequality for the related averaging operators, which implies that the set of functions enjoying pointwise convergence is closed. 
\end{itemize}
What is transferred between systems is this maximal inequality, with the most remarkable example being the Hardy--Littlewood maximal inequality for the one-sided averaging operators on $\ZZ$, which -- when combined with the mean ergodic theorem of von Neumann \cite{Neu} -- can be used to deduce the celebrated pointwise ergodic theorem of Birkhoff \cite{Bir}. 

In the modern approach, on the other hand, maximal estimates are replaced by some stronger ones -- where variation of a sequence of numbers rather than the maximum is controlled -- thanks to which pointwise convergence can be established for all functions directly. This idea was crucial in a breakthrough series of articles by Bourgain \cite{B1, B2, B3} who obtained the pointwise ergodic theorem for operators averaging along orbits with times determined by polynomials. Very recently, Krause, Mirek, and Tao \cite{KMT} were able to show a similar result for certain bilinear operators, also using the ideas of \cite{Ca} to reduce the problem to studying canonical systems. These results are major steps towards confirming the general conjecture about pointwise convergence of multilinear averages taken along polynomial orbits, see Conjecture~\ref{limAf}, promoted by Furstenberg and stated by Bergelson and Leibman in \cite{BL}.

In this note we shall show that the Calder\'on transference principle is valid for a~broad spectrum of scenarios regarding both, the types of underlying dynamical systems and the types of estimates to be transferred, see Theorem~\ref{CT}. In particular, all the settings considered in Conjecture~\ref{limAf} are captured. As we shall prove, the whole process can be carried out without increasing the optimal constants in the studied inequalities, see \eqref{sharpCT}.

In view of \eqref{sharpCT}, the following dichotomy arises naturally. For a given dynamical system the best constant in the studied inequality is either equal or strictly smaller than the best constant in the same inequality for the associated canonical system. To illustrate the importance of this observation, we shall use \eqref{sharpCT} to classify for which systems equality or strict inequality regarding the two constants occurs, in the case of weak type $(1,1)$ and strong type $(p,p)$ maximal inequalities for centered, uncentered or one-sided averages taken along linear orbits determined by a single ergodic transformation, see Theorem~\ref{T1}.      

\subsection{Calder\'on transference}

Let $\XX$ be a \emph{measure-preserving dynamical system}, that is, a quadruple $(X, \calB, \mu, \mathcal T)$, where $(X, \calB, \mu)$ is a nontrivial\footnote{Here by ``nontrivial'' we mean that there exists a subset of $X$ with finite and strictly positive measure.} $\sigma$-finite measure space and $\mathcal T = (T_1, \dots, T_d)$, $d \in \NN$, is a family of \emph{measure-preserving transformations} on $X$. The latter means that $T_i \colon X \to X$ are measurable and $\mu(T_i^{-1}(E)) = \mu(E)$ for all $E \in \calB$. Throughout the paper we assume that $T_i$ are invertible and commute with each other.

Next, for $m, k \in \NN$, let $\mathcal P = (P_{1,1}, \dots, P_{d,m} )$ be a family of $k$-variate polynomials $P_{i,j}$ such that $P_{i,j}(\ZZ^k) \subseteq \ZZ^k$. The associated $m$-linear averaging operators are defined by
\begin{equation} \label{Af}
A^{\mathcal P}_N f(x) \coloneqq
\mathbb E_{n \in [N]^k} \prod_{j \in [m]} f_j ( T_1^{P_{1,j}(n)} \cdots T_d^{P_{d,j}(n)} x ), \qquad x \in X.
\end{equation}
Here $f = (f_1, \dots, f_m)$ is an $m$-tuple of $\mu$-measurable functions $f_j \colon X \to \RR$, by $[l]$ we mean $\{1, \dots, l\}$, and $\mathbb E_{y \in Y} g(y)$ is the expected value of $g$, where the expectation is taken with respect to the discrete uniform distribution over the indicated finite set $Y$.

Regarding the operators \eqref{Af}, the following conjecture was posed in \cite{BL}.

\begin{conjecture} \label{limAf}
Fix $d,m,k \in \NN$ and let $\mathbb X, \mathcal P$ be as before. Then, for the operators \eqref{Af},  
\begin{equation*} 
\lim_{N\to \infty} A^{\mathcal P}_N f(x)
\end{equation*}
exists $\mu$-almost everywhere for each $m$-tuple $f = (f_1, \dots, f_m)$ with $f_j \in L^\infty(\mathbb X)$, $j \in [m]$. 
\end{conjecture}

As mentioned before, the conjecture was verified in some specific cases and the proofs usually required some quantitative knowledge about the behavior of $(A^{\mathcal P}_N f(x))_{N \in \NN}$. It is very likely that the general case, if true, will be proved in a similar fashion.

With this in mind we introduce an abstract map $\mathcal O \colon \RR^{\NN} \to [0, \infty]$ whose arguments are sequences of real numbers. The only assumption imposed on $\mathcal O$ will be that it is approximated from below by a sequence of measurable mappings $(\mathcal O_K)_{K \in \NN}$ depending on finitely many coordinates. Precisely, for each $K \in \NN$ let $\mathcal O_K \colon \RR^K \to [0, \infty)$ be measurable with respect to the standard topology on $\RR^K$. We assume that $\mathcal O_K$ can be chosen so that
\begin{equation} \label{O1}
\mathcal O (a_n : n \in \NN) \geq \mathcal O_K (a_n : n \in [K]), \qquad K \in \NN, \, (a_n : n \in \NN) \in \RR^{\NN}
\end{equation}
and
\begin{equation} \label{O2}
\mathcal O (a_n : n \in \NN) = \lim_{K \to \infty} \mathcal O_K (a_n : n \in [K]), \qquad (a_n : n \in \NN) \in \RR^{\NN}.
\end{equation} 

Given $p = (p_0,p_1, \dots, p_m) \in (0, \infty)^{m+1}$ with $\frac{1}{p_1} + \cdots + \frac{1}{p_m} = \frac{1}{p_0}$, we examine the inequalities
\begin{equation} \label{INQ}
\| \mathcal O (A_N^{\mathcal P}f : N \in \NN) \|_{p_0}
\leq C \prod_{j \in [m]} \| f_j \|_{p_j}
\quad \text{and} \quad 
\| \mathcal O (A_N^{\mathcal P}f : N \in \NN) \|_{p_0, \infty}
\leq C \prod_{j \in [m]} \| f_j \|_{p_j},
\end{equation} 
where by $\mathcal O(g_n : n \in \NN)$ we mean the function $x \mapsto \mathcal O(g_n(x) : n \in \NN)$, while $\|g\|_{q}$ and $\|g\|_{q,\infty}$ stand for the usual Lebesgue and weak Lebesgue $q$-quasinorms of $g$ with respect to $\mu$.
More precisely, given $g \colon X \to \RR$ and $q \in (0,\infty)$, we have
\[
\|g\|_{q} \coloneqq \Big (  \int_X |g(x)|^q \, {\rm d}\mu(x) \Big )^{1/q}
\quad \text{and} \quad
\|g\|_{q, \infty} \coloneqq \sup_{\lambda \in (0, \infty)} \lambda \mu( \{ x \in X : |g(x)| \geq \lambda \})^{1/q}.
\]  
We write respectively $\mathcal C_{\mathcal O}^{\mathcal P}(\XX, p, {\rm s})$ and $\mathcal C_{\mathcal O}^{\mathcal P}(\XX, p, {\rm w})$ for the smallest constants $C \in [0,\infty]$ such that these inequalities hold true for all $m$-tuples $f \in L^{p_1}(\mathbb X) \times \cdots \times L^{p_m}(\mathbb X)$.

While reading, one can think of a model case, where $\mathcal O$ is the supremum norm 
\[
\mathcal M (a_n : n \in \NN) \coloneqq \sup_{n \in \NN} |a_n|.
\]
Notice that $\mathcal M$ is approximated from below by $\mathcal M_K(a_n : n \in [K]) \coloneqq \max_{n \in [K]} |a_n|$. Another important example is the $r$-variation seminorm, $r \in [1, \infty)$, given by 
\[
\mathcal V^r (a_n : n \in \NN) \coloneqq \sup_{J \in \NN} \sup_{n_0 < n_1 < \cdots < n_{J}} \Big( \sum_{j \in [J]} |a_{n_j} - a_{n_{j-1}}|^r \Big)^{1/r}
\]
approximated by its truncated versions $\mathcal V^r_K$, where we additionally demand $n_J \leq K$. 

Let us observe that $\mathcal M$ and $\mathcal V^r$ control respectively the size and variability of the input values. This is why estimates for $\mathcal V^r$ can be used to prove pointwise convergence directly, while maximal inequalities are only effective in estimating error terms, therefore requiring a dense class of functions for which pointwise convergence is known {\it a priori}.

In practice, estimates for $\mathcal V^r$ are more difficult to prove than their counterparts for $\mathcal M$. Moreover, quite often they hold only in a limited range. For example, sequences of expected value operators coming from martingales enjoy variation estimates for $r \in (2,\infty)$ but not for $r \in [1,2]$, see \cite{Lep, PX}. For more details and other important examples related to oscillation or jump inequalities we refer the reader to \cite{JSW, MSS, Slo}. 

Regarding the underlying system $\XX$, our reference point is the $d$-dimensional canonical system $\XX_d = (\ZZ^d, 2^{\ZZ^d}, \#_d, \mathcal T_d)$, that is, $\ZZ^d$ equipped with the $\sigma$-algebra of all subsets, counting measure, and the family $\mathcal T_d = (\mathcal T_{d,1}, \dots, \mathcal T_{d,d})$ of $d$ independent shifts defined by 
\[
\mathcal T_{d,i}(l_1, \dots, l_{i-1}, l_i, l_{i+1}, \dots l_d) \coloneqq
(l_1, \dots, l_{i-1}, l_i+1, l_{i+1}, \dots l_d).
\]

We shall show the following result relating the actions of $\mathcal O$ on $\XX$ and $\XX_d$ to each other.
\begin{theorem}[Calder\'on transference] \label{CT}
	Fix $d,m,k \in \NN$ and let $\mathbb X, \mathcal P, p$ be as before. Then
	\begin{equation} \label{sharpCT}
	\mathcal C_{\mathcal O}^{\mathcal P}(\XX, p, {\rm s}) \leq \mathcal C_{\mathcal O}^{\mathcal P}(\XX_d, p, {\rm s})
	\quad \text{and} \quad
	\mathcal C_{\mathcal O}^{\mathcal P}(\XX, p, {\rm w}) \leq \mathcal C_{\mathcal O}^{\mathcal P}(\XX_d, p, {\rm w}) 
	\end{equation}
	hold for all $\mathcal O$ admitting \eqref{O1} and \eqref{O2}, where $\XX_d$ is the $d$-dimensional canonical system. 
\end{theorem} 

\noindent In particular, regarding maximal, variation, oscillation or jump inequalities it is enough to deal with canonical systems.  

\subsection{Dichotomy} Theorem~\ref{CT} motivates the following question:
\[
\textit{When does the equality take place in \eqref{sharpCT}?} 
\]
We shall show how to deal with such a kind of problem, by solving it in a very particular relatively easy case. In this subsection we take $d = m = k = 1$ and $P(n) = n$. Moreover, we assume that   
$T$ is {\it ergodic}, which means that $T^{-1}(E) = E$ implies $\mu(E) =0$ or $\mu(X \setminus E)=0$. 

In this context we introduce the one-sided maximal operator $\TOS$ by using the formula
\[
\TOS f(x) \coloneqq \sup_{N \in \NN \, \cup \{0\}} \Big| \frac{1}{N+1} \sum_{n=0}^{N} f(T^n x) \Big|, \qquad x \in X. 
\] 
Similarly, we define the two-sided maximal operators, centered $\TC$ and uncentered $\TNC$, by
\[
\TC f(x) \coloneqq \sup_{N \in \NN \, \cup \{0\} } \Big|\frac{1}{2N+1} \sum_{n=-N}^{N} f(T^n x)\Big| \quad \text{and} \quad \TNC f(x) \coloneqq \sup_{\underline r, \overline r \in \NN \cup \{0\} } \Big|\frac{1}{\underline r + \overline r+1} \sum_{n=-\underline r}^{\overline r} f(T^n x)\Big|.
\]

Regarding these operators, it is natural to ask about the weak type $(1,1)$ and strong type $(p,p)$ inequalities for $p \in (1, \infty]$. 
Thus, we put $\COS(\XX, 1) \coloneqq \mathcal C_{\TOS}^P(\mathbb X, (1,1), {\rm w})$ and similarly $\COS(\XX, p) \coloneqq \mathcal C_{\TOS}^P(\mathbb X, (p,p), {\rm s})$ for each $p \in (1, \infty)$, while $\COS(\XX, \infty)$ will stand for the best constant $C$ in the inequality $ \| \TOS f \|_\infty \leq C \| f \|_\infty$, where 
\[\| g \|_\infty \coloneqq \inf \{ \lambda \in [0,\infty] : |g(x)| \leq \lambda \text{ for } \mu\text{-almost every } x \in X \}.
\]   
We define $\CCC(\XX, p)$ and $\CNC(\XX, p)$, $p \in [1,\infty]$, in the same manner, replacing $\TOS$ with $\TC$ and $\TNC$. In view of Theorem~\ref{CT}, properties of the Hardy--Littlewood maximal operators on $\ZZ$, and obvious bounds for $p = \infty$, all these quantities are finite regardless of $\XX$.

Given a nontrivial ergodic system $\XX$, we say that $X$ consists of finitely many atoms if it splits into disjoint measurable sets $X_0, X_1, \dots, X_L$, $L \in \NN$, such that $\mu(X_0) = 0$, $T(X_l) \subseteq X_{l+1}$ for $l \in [L-1]$, $T(X_L) \subseteq X_1$, and none of $X_l$, $l \in [L]$, can be split further into two disjoint sets of nonzero measure. Observe that then $\mu(X_1) = \cdots = \mu(X_L) \in (0,\infty)$. 
   
The following well-known fact will be useful later on.
\begin{fact} [Kakutani--Rokhlin lemma] \label{F1}
	Let $\XX$ be a nontrivial ergodic system. Then exactly one of the following two possibilities holds:
	\begin{enumerate}[label=(\Alph*)]
		\item \label{F1A} the set $X$ consists of finitely many atoms, 
		\item \label{F1B} for each $L \in \NN$ there exists $E_L \in \calB$ such that $\mu(E_L) \in (0, \infty)$ and the sets $T^{-l}(E_L)$, $l \in [L]$, are disjoint. 
	\end{enumerate}
\end{fact}

We are ready to formulate our second main result.
\begin{theorem} [Sharp constants dichotomy] \label{T1}
	Let $\XX$ be as before. Then for $p \in \{1, \infty\}$ we have $\COS(\XX, p) = \COS(\XX_1, p)$, while for $p \in (1,\infty)$ the following dichotomy occurs:
	\begin{itemize}
		\item  if \ref{F1A} from Fact~\ref{F1} holds, then $\COS(\XX, p) < \COS(\XX_1, p)$, 
		\item  if \ref{F1B} from Fact~\ref{F1} holds, then $\COS(\XX, p) = \COS(\XX_1, p)$.
	\end{itemize}
	Moreover, we have $\CCC(\XX, \infty) = \CCC(\XX_1, \infty)$ and $\CNC(\XX, \infty) = \CNC(\XX_1, \infty)$, while for $p \in [1,\infty)$ the following dichotomy occurs:
	\begin{itemize}
		\item 	if \ref{F1A} from Fact~\ref{F1} holds, then $\CCC(\XX, p) < \CCC(\XX_1, p)$ and $\CNC(\XX, p) < \CNC(\XX_1, p)$,
		\item 	if \ref{F1B} from Fact~\ref{F1} holds, then $\CCC(\XX, p) = \CCC(\XX_1, p)$ and $\CNC(\XX, p) = \CNC(\XX_1, p)$.
	\end{itemize}
\end{theorem}
This result has interesting consequences because some of the constants on $\XX_1$ are known.
\begin{corollary}\label{C1}
	Let $\XX$ be as before and assume that \ref{F1B} from Fact~\ref{F1} holds. Then 
	\begin{itemize}
		\item for $p=1$ we have $\COS(\XX, 1)=1$, $\CNC(\XX, 1)=2$, and $\CCC(\XX, 1) = \frac{11+\sqrt{61}}{12}$, 
		\item for $p\in (1, \infty)$ and the uncentered operator we have $\CNC(\XX, p) = c_p$, where $c_p$ is the unique positive solution of the equation $(p-1)x^p - p x^{p-1} - 1 = 0$,
		\item for $p=\infty$ we have $\COS(\XX, \infty)=\CNC(\XX, \infty)=\CCC(\XX, \infty) = 1$.
	\end{itemize} 
\end{corollary}
\noindent Indeed, the equality $\COS(\XX_1, 1)=1$ is well-known, see Lemma~\ref{L1} for an easy proof. Next, $\CNC(\XX_1, 1)=2$ because $\CNC(\XX_1, 1) \leq 2$ follows from the classical covering lemma on $\ZZ$ with overlap $2$, while $\CNC(\XX_1, 1) \geq 2$ can be seen by taking $\tilde{f} = \ind{\{0\}}$ and letting $\lambda \to 0$. On the other hand, $\CCC(\XX_1, 1) = \frac{11+\sqrt{61}}{12}$ follows by combining the famous result of Melas \cite{Mel}, where the best constant in the weak type $(1,1)$ inequality for the centered maximal operator on $\RR$ is obtained, with \cite[Theorem~1]{KMPW}, where the best constants in discrete and continuous inequalities are compared to each other (in our case both constants are equal). For $p = \infty$ the claim is trivial, and for $p \in (1, \infty)$, just like before, we combine the result on $\RR$ of Grafakos and Montgomery-Smith \cite{GMS}  with the following transference principle.

\begin{proposition}\label{P2}
	Fix $p \in (1, \infty)$ and let $\XX$ be as before. Then $\CNC(\XX, p) = c_p$, where $c_p$ is the best constant in the strong type $(p,p)$ inequality for $\calM_\RR^{\rm u}$, the uncentered Hardy--Littlewood maximal operator $\RR$. 
\end{proposition}

\begin{proof}
	The inequality $\CNC(\XX, p) \geq c_p$ follows by repeating the arguments used to prove \cite[(1.3)]{KMPW}. We only sketch the proof. Let $f_{\rm cont} \colon \RR \to [0,\infty)$ be a smooth function for which $c_p$ is almost attained. Using dilations we can assume that $f_{\rm cont}$ is very slowly varying. Define $f_{\rm dis} \colon \ZZ \to [0, \infty)$ by sampling $f_{\rm dis}(l) \coloneqq f_{\rm cont}(l)$. Then the constant in the maximal inequality stated for $f_{\rm dis}$ will be almost the same as the one corresponding to $f_{\rm cont}$. 
	
	On the other hand, $\CNC(\XX, p) \leq c_p$ follows easily, since for any $f_{\rm dis} \colon \ZZ \rightarrow [0,\infty)$ we have
	\[
	\frac{\|\TNC f_{\rm dis}\|_p}{\|f_{\rm dis}\|_p} \leq \frac{\|\calM_\RR^{\rm u} f_{\rm cont}\|_p}{\|f_{\rm cont}\|_p},
	\] 
	where $f_{\rm cont} \colon \RR \rightarrow [0,\infty)$ is given by $f_{\rm cont} (x) \coloneqq f_{\rm dis}(l)$ for $x \in [l, l+1), l \in \NN$. Indeed, this inequality can be easily seen, since for each $x \in [l, l+1)$ one has $\calM_\RR^{\rm u} f_{\rm cont}(x) \geq \TNC f_{\rm dis}(l)$ (it is important here that we are working with operators of uncentered type).
\end{proof}

\noindent Similar arguments give an analogue of Proposition~\ref{P2} for the one-sided operator. For the centered operator the situation is different and only the lower bound for $\CCC(\XX, p)$ can be easily transferred from the continuous setting to the discrete one, cf.~\cite[Theorem~1]{KMPW}.

\subsection{Ergodicity} \label{S1.4}
The canonical system $\mathbb X_d$ can be viewed as a topological  group with Haar measure and transformations $\mathcal T_{d,i}$ being translations by group elements, say $g_{d,i}$. Since the subgroup $G$ generated by $g_{d,i}$ is just $\ZZ^d$, the canonical system is ergodic, that is, there are no nontrivial $G$-invariant sets $E \subseteq \mathbb Z^d$. We take this opportunity to highlight the importance of ergodicity in this context.

Let us specify our measure space $(X,\mathcal B, \mu)$ to be $[1,0)$ with the usual Borel sets and Lebesgue measure. We consider the group $\rm Aut(X,\mathcal B, \mu)$ of bimeasurable automorphisms of that space, that is, invertible maps $T \colon X \to X$ such that $T,T^{-1}$ are measure-preserving.  In fact, elements of $\rm Aut(X,\mathcal B, \mu)$ are equivalence classes, with $T \sim T'$ if and only if $T(x) = T'(x)$ holds $\mu$-almost everywhere, and we shall refer to their representatives.

We say that $(T_n)_{n \in \NN}$ \emph{converges weakly} to $T$ if $\lim_{n \to \infty} \mu(T_n(E)\Delta T(E)) = 0$ for all $E \in \mathcal B$ with $\Delta$ being the symmetric difference symbol. Also, if $\mu(\{ x \in [0,1) : T^n(x) = x\})=0$ for all $n \in \NN$, then $T$ is called \emph{aperiodic}. Following Halmos \cite{Ha} we notice that for aperiodic $T$ its conjugacy orbit $\{S^{-1} T S : S \in \rm Aut(X,\mathcal B, \mu)\}$ is weakly dense. If $T$ is ergodic instead, then so are all $S^{-1}TS$, while properties of $\mu$ imply that $T$ is aperiodic as well\footnote{Note that ergodic systems satisfying \ref{F1A} from Fact~\ref{F1} lack aperiodicity so referring to $\mu$ is necessary.}. In view of ergodicity of $T_\pi(x) \coloneqq x + \pi \mod 1$, say, ergodic automorphisms are weakly dense. This fact alone suggests that many problems are reducible to the case of ergodic transformations.

Yet another relevant result is the so-called Conze principle \cite{Co}. Its original formulation is as follows. Given a sequence $(i_n)_{n \in \NN}$ of integers, for each $S \in \rm Aut(X,\mathcal B, \mu)$ the sharp constant in the weak type $(1,1)$ inequality for $S_\ast f(x) \coloneqq \sup_{ N \in \NN} |\frac{1}{N} \sum_{n=1}^N f (S^{i_n} x)|$ is not larger than its counterpart for $T_\ast$ defined analogously for any aperiodic $T \in \rm Aut(X,\mathcal B, \mu)$. The proof uses weak density of the conjugacy orbit of $T$, and the principle can be generalized in many ways. Again, by properties of $\mu$ the same is true for any ergodic $T$ instead.

The last related issue we would like to discuss is the Stein maximal principle \cite{St}. As we have seen, pointwise convergence results are often consequences of weak type inequalities, while proving the latter can be reduced to the ergodic case whenever the Conze principle is available. Stein's result, saying that in some cases pointwise convergence is equivalent to the associated weak type inequality, makes these observations more fundamental. 

Let this time $X$ be a~homogeneous space of a given compact group $G$. By this we mean that, in particular, the structure of $X$ is inherited from $G$, there is a unique normalized $G$-invariant measure $\mu$ on $X$, and $G$ acts transitively on $X$. For example, $X$ may be the $n$-dimensional sphere with $G$ consisting of all rotations of $X$. Transitivity implies that there are no nontrivial $G$-invariant sets $E \subseteq X$, and ergodicity indeed is crucial here. 

Now Stein's result reads as follows. For each $p \in [1,2]$ and each given sequence $(T_n)_{n \in \NN}$ of operators bounded on $L^p(X, \mu)$ and commuting with all transformations determined by the actions of $g \in G$ on $X$, if $\lim_{n \to \infty} T_n f(x)$ exists almost everywhere for all $f \in L^p(X, \mu)$, then $T_\ast f(x) \coloneqq \sup_{n \in \NN} |T_n f(x)|$ satisfies the weak type $(p,p)$ inequality. 

Let us check that ergodicity cannot be dropped. To see this, suppose that $X$ consists of disjoint $G$-invariant sets $X_l$, $l \in \NN$, admitting Stein's principle and such that $\mu(X_l) = 2^{-l}$. Then pointwise convergence gives weak type inequalities on all $X_l$ separately but the same for $X$ may be false, as the best constants may tend to infinity with $l$. This happens when $T_n f(x)$ equals $nf(x)$ for $x \in X_n$ and $0$ otherwise.

For another instructive example, take $X_l \coloneqq \{(l,0), (l,1), \dots, (l,l) \}$ with equal masses $2^{-l}/(l+1)$, and set $G \coloneqq \ZZ_2 \times \ZZ_3 \times \cdots$ with the action of $g=(g_1, g_2, \dots) \in G$ on $X$ given by $\tau_g(l,i) \coloneqq (l, i+g_l \mod l+1)$. Then for $g_\ast \coloneqq (1,1,\dots) \in G$ and $T_n f(x) \coloneqq \frac{1}{n} \sum_{i=1}^n f(\tau_{g_\ast}^{i^2}x)$ we see that $\lim_{n \to \infty} T_n f(x)$ exists everywhere for all $f \in L^1(X,\mu)$, while there is no weak type $(1,1)$ inequality in view of the famous result by Buczolich and Mauldin \cite{BM}.                     

\subsection{Acknowledgments} The author would like to thank Mariusz Mirek for drawing his attention to the topic discussed in this article. Moreover, the author is grateful to Mariusz Mirek and Luz Roncal for their helpful comments after reading the manuscript.

The author is also thankful to the referee for insightful comments on ergodicity, which resulted in a fruitful discussion in Subsection~\ref{S1.4}.   

The author was supported by the Basque Government (BERC 2022-2025), by the Spanish State Research Agency
(CEX2021-001142-S and RYC2021-031981-I), and by the Foundation for Polish Science (START 032.2022).  

\section{Proof of Theorem~\ref{CT}} \label{S2}

Let us now prove Theorem~\ref{CT}.

\begin{proof}[Proof of Theorem~\ref{CT}]
We only prove the first inequality and the second one can proved by using very similar arguments.

Let $f = (f_1, \dots, f_m)$ with $f_j \in L^{p_j}(\mathbb X)$. For each $K \in \NN$, we define 
\[
F_j^K (x, l) \coloneqq f_j ( T_1^{- KR_K + l_1} \cdots T_d^{- KR_K + l_d} x )
\cdot \ind{ \{ l \in [K R_K]^d \} },
\qquad x \in X, \, l \in \ZZ^d, 
\]
with $R_K \in \NN$ is so large that $|P_{i,j}(l)| \leq R_K$ for all $i,j$ when $\| l \|_\infty \leq K$. Then
\[
\| F_j^K \|_{p_j} = (KR_K)^{d/p_j} \|f_j\|_{p_j},
\]
by the Fubini--Tonelli theorem, since $T_i$ are measure preserving. The first norm above is with respect to the product of $\mu$ and counting measure. Set $F^K = (F_1^K, \dots, F_m^K)$. We have
\[
\mathcal O_K (A^{\mathcal P}_{N} F^K(x,l) : N \in [K]) 
=
\mathcal O_K (A^{\mathcal P}_{N} f(T_1^{- KR_K + l_1} \cdots T_d^{- KR_K + l_d} x) : N \in [K]) 
\]
for all $x \in X$ and $l \in \{ R_K+1, \ldots, (K-1)R_K \}^d$ when $K \geq 3$. Hence
\[
\| \mathcal O_K (A^{\mathcal P}_{N} F^K : N \in [K]) \|_{p_0}
\geq ((K-2)R_K)^{d/p_0} \| \mathcal O_K (A^{\mathcal P}_{N} f : N \in [K]) \|_{p_0}
\]
by the Fubini--Tonelli theorem, since $T_i$ are measure preserving. Letting $K \to \infty$ gives
\[
\| \mathcal O (A^{\mathcal P}_{N} f : N \in \NN) \|_{p_0}
\leq
\mathcal C_{\mathcal O}^{\mathcal P}(\XX_d, p, {\rm s})
\lim_{K \to \infty}
((K-2)R_K)^{-d/p_0}
\prod_{j \in [m]} \| F_j^K\|_{p_j}
\]
by Fatou's lemma, \eqref{O1}, \eqref{O2}, and the definition of $\mathcal C_{\mathcal O}^{\mathcal P}(\XX_d, p, {\rm s})$. The limit is equal to
\[
\lim_{K \to \infty}
((K-2)R_K)^{-d/p_0} (K R_K)^{d/p_1 + \cdots + d/p_m}
\prod_{j \in [m]} \| f_j \|_{p_j}
=
\prod_{j \in [m]} \| f_j \|_{p_j}
\]
thanks to the H\"older exponent hypothesis. This completes the proof. 
\end{proof}

Once we have seen the proof of Theorem~\ref{CT}, several remarks are in order.
\begin{itemize}
	\item Apart from those specified in \eqref{INQ}, other types of estimates could be considered in Theorem~\ref{CT}. One of the key factors is good scaling which in our case was assured by assuming $\frac{1}{p_1} + \cdots + \frac{1}{p_m} = \frac{1}{p_0}$.    
	\item Properties of $\XX_d$ resemble the structure of the group of transformations generated by $T_i$. Dropping commutativity is possible but it leads to more complicated canonical systems, see \cite{IMMS} for the case of nilpotent groups of step $2$.    
	\item Instead of $A_N^{\mathcal P}$, one can consider different operators, for example the averaging operators used to define $\TOS, \TC, \TNC$ from Theorem~\ref{T1}, see Proposition~\ref{P1}. 
\end{itemize} 

\begin{proposition} \label{P1}
	Under the assumptions of Theorem~\ref{T1}, for each $p \in [1,\infty]$ we have $\COS(\XX, p) \leq \COS(\XX_1, p)$, $\CCC(\XX, p) \leq \CCC(\XX_1, p)$, and $\CNC(\XX, p) \leq \CNC(\XX_1, p)$.
\end{proposition}
\begin{proof}
	We repeat the proof of Theorem~\ref{CT}.
 \end{proof}

\section{Proof of Theorem~\ref{T1}} \label{S3}
 
The proof of Theorem \ref{T1} is divided into several lemmas. Throughout this section $d=m=k=1$, $T$ is ergodic, and $P(n)=n$.
 
\begin{lemma}[Sharp constant equals $1$] \label{L1}
We have $\COS(\XX_1, 1) = \COS(\XX_1, \infty) = \CCC(\XX_1, \infty) = \CNC(\XX_1, \infty) = 1$. Consequently, $\COS(\XX, 1) = \COS(\XX_1, 1)$, $\COS(\XX, \infty) = \COS(\XX_1, \infty)$, $\CCC(\XX, \infty) = \CCC(\XX_1, \infty)$, and $\CNC(\XX, \infty) = \CNC(\XX_1, \infty)$. 
\end{lemma}
\begin{proof}
We only need to prove the first part and the rest of the claim follows directly from Proposition~\ref{P1} and the fact that $|f| \leq \min \{ \TOS f, \TC f, \TNC f\}$. The case $p = \infty$ is obvious. Also $\COS(\XX_1, 1) = 1$ is well-known but we include the proof for the sake of completeness. 

We shall show that $\COS(\XX_1, 1) \leq 1$ and the remaining inequality $\COS(\XX_1, 1) \geq 1$ is obvious. Let $f \in L^1(\XX_1)$ be nonnegative and assume that $\lambda \in (0,\infty)$ is such that the set
\[
\EOS(f) \coloneqq \{ l \in \ZZ : |\TOS f(l)| \geq \lambda \}
\] 
is nonempty. Thus, $\EOS(f) = \{ l_1, \dots, l_J\}$ for some integers $l_1 < \dots < l_J$, $J \in \NN$. For each $j \in [J]$ we let $N_j$ be the smallest number $N \in \NN \cup \{0\}$ such that $\frac{1}{N+1} \sum_{n=0}^{N} f(l_j+n) \geq \lambda$. We set $B_j \coloneqq \{l_j, \ldots, l_j+N_j \}$ and observe that $ \EOS(f) \subseteq \bigcup_{j=1}^J B_j$. Next, define $D_1 \coloneqq B_1$ and $D_j \coloneqq B_j \setminus \bigcup_{i=1}^{j-1} B_i$ for $1 \neq j \in [J]$. Then $D_j$ are disjoint, we have $\bigcup_{j=1}^J B_j = \bigcup_{j=1}^J D_j$, and for each $j$ either $D_j = \emptyset$ or $D_j = \{ l_j+r_j, \ldots, l_j+N_j \}$ for some $r_j \in [N_j] \cup \{0\}$. In the latter case, by the definition of $N_j$, we have $\sum_{l \in D_j} f(l) \geq \lambda \, |D_j|$. Consequently,
\[
\lambda \, |\EOS(f)| \leq \sum_{j \in [J]} \lambda \, |D_j| \leq \sum_{j \in [J]} \sum_{l \in D_j} f(l) \leq \|f\|_1
\]
which justifies $\COS(\XX_1, 1) \leq 1$. 
\end{proof}

Let us now consider the case $p \in (1, \infty)$ for atomic systems.

\begin{lemma}[Case~\ref{F1A}, $p \in (1,\infty)$] \label{L2}
Assume that \ref{F1A} from Fact~\ref{F1} holds. Then for each $p \in (1, \infty)$ we have $\COS(\XX, p) < \COS(\XX_1, p)$, $\CCC(\XX, p) < \CCC(\XX_1, p)$, and $\CNC(\XX, p) < \CNC(\XX_1, p)$.
\end{lemma}
\begin{proof}
We only prove $\COS(\XX, p) < \COS(\XX_1, p)$ and the remaining inequalities can be proved by using very similar arguments.

It suffices to show that for each fixed $L \in \NN$ and $p \in (1,\infty)$ we have $\COS(\XX_{[L]}, p) < \COS(\XX_1, p)$, where $\XX_{[L]} = ([L], 2^{[L]}, \#_{[L]}, T_{[L]})$ is the finite shift system with $L$ elements, that is, $T_{[L]}(l) \coloneqq l+1$ for $l \in [L-1]$ and $T_{[L]}(L) \coloneqq 1$.

Given $\epsilon \in (0,1)$ let $f \colon [L] \to [0,\infty)$ be such that $\| \TOS f \|_p \geq (1-\epsilon) \COS(\XX_{[L]}, p) \|f\|_p$. Then for a large parameter $R \in \NN$ we define $F \colon \ZZ \to [0,\infty)$ by
\[
F(l) \coloneqq \left\{ \begin{array}{rl}
f(l \ {\rm mod} \ L) & \textrm{for } l \in [RL], \\
0 & \textrm{otherwise}, \end{array} \right.
\]
where we identify $f(0)$ with $f(L)$.
Note that $\|F\|_p=R^{1/p} \|f\|_p$. Moreover, if $R$ is sufficiently large (with respect to $L$, $p$, and $\epsilon$), then 
\[
\| \TOS F \cdot \ind{[RL]} \|_p \geq (1-\epsilon)^2 \COS(\XX_{[L]}, p) \|F\|_p.
\]
For each $l \in \{-RL+1,\ldots, 0\}$, by using H\"older's inequality, we obtain
\[
\TOS F(l) \geq \frac{1}{2RL} \sum_{n \in [RL]} F(n) = \frac{1}{2L} \sum_{n \in [L]} f(n) \geq \frac{\|\TOS f\|_\infty}{2L} \geq \frac{\|\TOS f\|_p}{2 L^{1+1/p}} \geq \frac{(1-\epsilon) \COS(\XX_{[L]}, p)}{2 L^{1+1/p} R^{1/p}} \|F\|_p. 
\] 
Consequently,
\[
\| \TOS F \cdot \ind{\{-RL+1,\ldots, RL\}} \|_p^p \geq \Big( (1-\epsilon)^{2p} +
RL \cdot \frac{(1-\epsilon)^p}{2^p L^{p+1}R}
\Big) \, \COS(\XX_{[L]}, p)^p \|F\|_p^p > \COS(\XX_{[L]}, p)^p \|F\|_p^p,
\]
provided that $\epsilon$ sufficiently small (with respect to $L$ and $p$).
\end{proof}

Next we consider the case $p \in [1, \infty)$ for other systems.

\begin{lemma}[Case~\ref{F1B}, $p \in [1, \infty)$] \label{L5}
Assume that \ref{F1B} from Fact~\ref{F1} holds. Then for each $p \in [1, \infty)$ we have $\COS(\XX, p) = \COS(\XX_1, p)$, $\CCC(\XX, p) = \CCC(\XX_1, p)$, and $\CNC(\XX, p) = \CNC(\XX_1, p)$.
\end{lemma}
\begin{proof}
	We only prove $\COS(\XX, p) = \COS(\XX_1, p)$ for $p \in (1, \infty)$ and the remaining equalities or the case $p=1$ can be verified by using very similar arguments.
	
 Given $\epsilon \in (0,1)$ one can find $L \in \NN$ and $F \colon \ZZ \to [0, \infty)$ such that $F(l)=0$ for each $l \in \ZZ \setminus [L]$ and
 \[
 \| \TOS F \cdot \ind{[L]} \|_p \geq (1-\epsilon) \COS(\XX_1, p) \|F\|_p.
 \]
 Choose $E_L \subseteq X$ as in \ref{F1B} from Fact~\ref{F1}, and define $f  \in L^p(\XX)$ by
 \[
 f(x) \coloneqq \left\{ \begin{array}{rl}
 F(L+1-l) & \textrm{for } x \in T^{-l}(E_L), \ l \in [L],\\
 0 & \textrm{otherwise}. \end{array} \right.
 \]
 Then it is easy to see that $\|\TOS f \|_p \geq (1-\epsilon) \COS(\XX_1, p) \|f\|_p$.  
\end{proof}

It remains to consider the case $p=1$ for atomic systems. We only focus on $\TC$ and $\TNC$, since the claim for $\TOS$ has already been proven in Lemma~\ref{L1}.

\begin{lemma}[Case~\ref{F1A}, $p = 1$, centered operator] \label{L6}
	Assume that \ref{F1A} from Fact~\ref{F1} holds. Then we have $\CCC(\XX, 1) < \CCC(\XX_1, 1)$.
\end{lemma}

\begin{proof}
	It suffices to show that for each fixed $L \in \NN$ we have $\CCC(\XX_{[L]}, 1) < \CCC(\XX_1, 1)$, where $\XX_{[L]}$ is as in Lemma~\ref{L2}. We follow the proof of \cite[Theorem 3]{Mel}.
	
	First observe that there exists $f_0 \colon [L] \to [0,\infty)$ such that $\| \TC f_0 \|_{1, \infty} = \CCC(\XX_{[L]}, 1) \| f_0 \|_1$. Indeed, let $(f_j)_{j \in \NN}$ be a sequence of nonnegative functions such that
	\[
	\lim_{j \to \infty} \frac{\| \TC f_j \|_{1, \infty} }{ \|f_j\|_1} = \CCC(\XX_{[L]}, 1).
	\]
	By applying the appropriate translation and scaling we can assume that $f_j(1) = \| f_j \|_\infty = 1$ holds. Choose a subsequence $(j_i)_{i \in \NN}$ such that $\lim_{i \to \infty} f_{j_i}(l)$ exists for each $l \in [L]$. Then $f_0$ defined by $ f_0(l) \coloneqq \lim_{i \to \infty} f_{j_i}(l)$ is the function we are looking for. Similarly, observe that there exists $\lambda \in (0,1]$ such that $\lambda \, |\EC(f_0) | = \CCC(\XX_{[L]}, 1) \| f_0 \|_1$. We define $f \coloneqq f_0 / \lambda$ and note that $f(l) \in [0,L]$ holds for each $l \in [L]$ (if $f(l)>L$, then $\| \TC f \|_{1,\infty} < \|f\|_1$, which leads to a contradiction). Below we identify $f$ with its $L$-periodic extension.
	
	We have $E^{\rm c}_1(f) = \{ l_1, \dots, l_J\}$ for some $J \in [L]$ and $l_1 < \dots < l_J$. For each $j \in [J]$ we choose $N_j \in \NN \cup \{0\}$ such that $f(l_j-N_j) + \dots + f(l_j+N_j) \geq 2N_j+1$. Consider the set
	\[
	\calS \coloneqq \{ (x_1, \dots, x_{L}) \in [0,L]^L : x_{l_j-N_j} + \dots + x_{l_j + N_j} \geq 2N_j+1 {\rm \ for \ each \ } j \in [J] \}.
	\]
	Notice that $(f(1), \dots, f(L)) \in \calS$. Consequently, $\calS$ is a nonempty, compact, convex polyhedron contained in $\RR^L$. We introduce a linear operator $\Lambda \colon \RR^L \to \RR$ defined by
	\[
	\Lambda(x_1, \dots, x_{L}) \coloneqq x_1 + \dots + x_{L}.
	\]
	Then there exists a vertex of $\calS$, say $(x^*_1, \dots, x^*_{L})$, such that
	\[
	\Lambda(x^*_1, \dots, x^*_{L}) = \min \{ \Lambda(x_1, \dots, x_{L}) : (x_1, \dots, x_{L}) \in \calS\} \leq \Lambda(f(1), \dots, f(L)).
	\]
	We observe that $(x^*_1, \dots, x^*_{L})$, as a vertex of $\calS$, is the only solution of the linear system consisting of all those of the equations $x_1 = 0, \dots, x_{L}=0$, $x_1=L, \dots, x_{L}=L$, and 
	$x_{l_1-N_1} + \dots + x_{l_1 + N_1} = 2N_1+1, \dots, x_{l_J-N_J} + \dots + x_{l_J + N_J} = 2N_J+1$, 
	which are satisfied with $(x^*_1, \dots, x^*_{L})$ in place of $(x_1, \dots, x_{L})$. Applying a standard argument from the theory of linear systems, we obtain that $x^*_l \in \QQ$ for each $l \in [L]$.
	
	Consider $f^* \colon [L] \to [0, \infty)$ defined by $f^*(l) \coloneqq x^*_l$. Then we have $\| f^* \|_1 \leq \|f\|_1$ and $E^{\rm c}_1(f^*) \supseteq E^{\rm c}_1(f)$. Consequently, $|E_{f^*}| = \CCC(\XX_{[L]}, 1) \| f^* \|_1$, which implies $ \CCC(\XX_{[L]}, 1) \in \QQ$. By using Proposition~\ref{P1} and the fact that $ \CCC(\XX_{\ZZ}, 1) = \frac{11+\sqrt{61}}{12} \notin \QQ$ (see \cite{Mel} and \cite[Theorem~1]{KMPW}), we conclude that $ \CCC(\XX_{[L]}, 1) < \CCC(\XX_{\ZZ}, 1)$.
\end{proof}

\begin{lemma}[Case~\ref{F1A}, $p = 1$, uncentered operator] \label{L7}
	Assume that \ref{F1A} from Fact~\ref{F1} holds. Then we have $\CNC(\XX, 1) < \CNC(\XX_1, 1)$.
\end{lemma}

\begin{proof}
	It suffices to show that for each fixed $L \in \NN$ we have $\CNC(\XX_{[L]}, 1) < \CNC(\XX_1, 1)$, where $\XX_{[L]}$ is as in Lemma~\ref{L2}. We also recall the well-known fact that $\CNC(\XX_1, 1) = 2$.
	
	Take $f \colon [L] \to [0, \infty)$ and $\lambda \in (0,\infty)$ such that $\ENC(f)$ is nonempty. Again, we identify $f$ with its $L$-periodic extension. Then $\ENC(f) = E_{\rm left} \cup E_{\rm right}$, where  
	\begin{align*}
	E_{\rm left} & \coloneqq \{ l \in [L] : f(l-\underline r) + \dots + f(l) \geq (n+1) \lambda {\rm \ for \ some \ } \underline r \in \NN \cup \{0\} \}, \\
	E_{\rm right} & \coloneqq \{ l \in [L] : f(l) + \dots + f(l+\overline r) \geq (n+1) \lambda {\rm \ for \ some \ } \overline r \in \NN \cup \{0\} \}.
	\end{align*}
	Applying the result for the one-sided maximal operator we see that
	\[
	|E_{\rm left}| \leq \|f\|_1 / \lambda \quad {\rm and} \quad |E_{\rm right}| \leq \|f\|_1 / \lambda.
	\]
	Since $E_{\rm left} \cap E_{\rm right} \neq \emptyset$, we deduce that, under the conditions specified above, the following estimate using the ceiling function must be satisfied
	\[
	\lambda \, |\ENC(f)| \leq 
	\max_{E, E' \subseteq [L] : E \cap E' \neq \emptyset}  \frac{|E \cup E'|}{\max \{ |E|, |E'|\} } \|f\|_1 
	 = \frac{2 \lceil \frac{L}{2} \rceil -1}{\lceil \frac{L}{2} \rceil} \|f\|_1.   
	\]
	 Consequently, $\CNC(\XX_{[L]}, 1) \leq \frac{2 \lceil \frac{L}{2} \rceil -1}{\lceil \frac{L}{2} \rceil} < 2$, as desired. In fact, $\CNC(\XX_{[L]}, 1) = \frac{2 \lceil \frac{L}{2} \rceil -1}{\lceil \frac{L}{2} \rceil}$, which can be proven by taking $\tilde{f} \coloneqq \ind{\lceil \frac{L}{2}\rceil }$ and $\tilde{\lambda} \coloneqq \lceil \frac{L}{2}\rceil^{-1}$, and observing that $E^{\rm u}_{\tilde{\lambda}}(\tilde{f}) = [2 \lceil \frac{L}{2}\rceil - 1]$. 
\end{proof}

\end{document}